\documentclass[a4paper,12pt,openbib,reqno]{amsart}
\usepackage{amsmath,amsfonts,amssymb}
\usepackage{verbatim}
\usepackage{enumerate}
\usepackage{url}
\usepackage[colorlinks]{hyperref}
\usepackage[latin1]{inputenc}
\usepackage[english]{babel}
\usepackage{tikz}
\usepackage[margin=2.5cm]{geometry}
\def\N{\mathbb N}

\def\Z{\mathbb Z}
\def\C{\mathbb C}
\def\Q{\mathbb Q}

\def \F{\mathbb F}

\def\Fx{\mathbb{F}_q[x]}
\def\Fq{\mathbb{F}_q}
\def\Fp{\mathbb{F}_p}
\def\Fqn{\mathbb{F}_{q^n}}
\def\Fqk{\mathbb{F}_{q^k}}

\def\Fpk{\mathbb{F}_{p^k}}
\DeclareMathOperator{\tr}{Tr}

\theoremstyle{plain}
\newtheorem{theorem}{Theorem}[section]
\newtheorem{lemma}[theorem]{Lemma}
\newtheorem{definition}[theorem]{Definition}
\newtheorem{corollary}[theorem]{Corollary}

\newtheorem{remark}[theorem]{Remark}
\newtheorem{example}[theorem]{Example}

\def\qed{\hfill\hbox{$\square$}}

\theoremstyle{definition}

\author[J. Alves Oliveira]{Jos\'e Alves Oliveira}
\author[F. E. Brochero Mart\'{\i}nez]{F. E. Brochero Mart\'{\i}nez}
\address{
	Departamento de Matem\'{a}tica\\
	Universidade Federal de Minas Gerais\\
	UFMG\\
	Belo Horizonte, MG\\
	31270-901\\
	Brazil\\
}
\email{jose-alvesoliveira@hotmail.com}
\email{fbrocher@mat.ufmg.br }

\title{Permutation Binomials over Finite Fields}
\keywords{Permutation Polynomial,  Algebraic Curves, Hasse-Weil's bound, caracter sum, index}

\date{\today
}

\subjclass[2000]{ }

\subjclass[2010]{12E20 (primary) and 11T06(secondary)}

\begin{document}
	
	\begin{abstract}
	Let $\Fq$ denote the finite field with $q$ elements. In this paper we use the relationship between suitable polynomials and number of rational points on algebraic curves to give the exact number of elements $a\in\Fq$ for which the binomial $x^n(x^{\frac{q-1}{r}}+a)$ is a permutation polynomial in the cases $r=2$ and $r=3$.
	\end{abstract}
	
	\maketitle
	
	\section{Introduction}
	Let $\Fq$ denote the finite field with $q$ elements. A polynomial $f(x)\in\Fx$ is called a \textit{permutation polynomial} over $\Fq$ if the map $a\mapsto f(a)$ permutes the elements of $\Fq$. Important early contributions to the general theory are contained in Hermite \cite{Her} and Dickson \cite{Dic}. 
	In recent times, the study of permutation polynomial has intensified due to their applications in cryptography and coding theory \cite{apl1,apl3,apl4,apl2}, resulting in the emergence of several new classes of these types of polynomials. 
	
	Linear polynomials $ax+b$, with $a\in\Fq^*$, are the simplest class of permutation polynomials of $\Fq$. Other simple family of permutation polynomials are the monomials; The monomial $x^k$ permutes $\Fq$ if and only if $\gcd(k,q-1)=1$. Therefore, we have a simple condition in order to determine if a monomial is a permutation polynomial. 
Now, characterizing permutation binomials is harder and more interesting. 
These polynomials have simple form and they are important because their easy computability. Characterization of permutation binomials of the form $x(x^{\frac{q-1}{2}}+a)$ was made in Niederreiter and Robinson \cite{Rob} and others particular classes of binomials have been study by several other authors (e.g. \cite{Akb, ref2, Kim, ref3, Smal1, Wan10, Wang, ref1}).
	
	The existence of permutation polynomials and quantity of them with some characteristic have been extensively explored recently. For instance, in the case of monomials, there are $\varphi(q-1)$ values of $k$ for which $x^k$ permutes $\Fq$. 
The existence of binomials of the forms $x(x^\frac{q-1}{2}+a)$ and $x(x^\frac{q-1}{3}+a)$ for sufficiently large $q$  was shown by Carlitz \cite{Carl}. In \cite{Bhat}, given $t$ odd, the authors calculate the number of elements $a\in\F_{2^t}$ for which the binomial of the form $x(x^{\frac{2^{n}-1}{3}}+a)$ permutes $\F_{2^n}$, where $n=2^it$. 
	Finding the exact number of permutation binomials of the form $x^n(x^{\frac{q-1}{m}}+a)$ is still an open problem. Advances in the solution of this problem can be find in \cite{DWan}, where the author proved that  the number of such permutation binomials of the form $f_a(x)  = x(x^{(q-1)/m}+a)$ is estimated to $\frac{m!}{m^m}q+{O}(\sqrt{q})$.
	
	There is an interesting connection between permutation binomials and algebraic curves and this fact was used by Masuda and Zieve to refine the estimate for the number of permutation binomials, as we can see in the following theorem.
	
	\begin{theorem}\cite[Theorem $3.1$]{MaZi1}\label{teoMasuda} Let $n,k>0$  be integers with $\gcd(n,k,q-1)=1$, and suppose $q\ge 4$.  If $\gcd(k,q-1)>2q(\log\log q)/\log q$, then there exists $a\in \F_q^*$ such that $x^n(x^k+a)$ permutes $\F_q$. Further, letting $T$ denote the number of $a\in \F_q$ for which $x^n(x^k+a)$ permute $\F_q$, and writing $r:=(q-1)/\gcd(k,q-1)$, we have
		$$ 
		\frac{r !}{r^r} \left(q+1-\sqrt{q}M_r-(r+1) r^{r-1}\right) \leq T \leq \frac{r !}{r^r}\left(q+1+\sqrt{q}M_r\right) ,
		$$
		where $M_r:=r^{r+1}-2 r^{r}-r^{r-1}+2$.
	\end{theorem}

	In order to prove this result, they used the Hasse-Weil's bound to estimate the number of $\Fq$-rational points on a curve. The link between polynomials over finite fields and algebraic curves has been extensively used in counting results. 
	In this article we also use this relationship, i.e. we use the connection between permutation binomials and algebraic curves to calculate the exact number of binomials of the forms $x^n(x^{\frac{q-1}{2}}+a)$ and $x^n(x^{\frac{q-1}{3}}+a)$.
In this paper we  relate the number of permutation binomials of the form $x^n(x^{\frac{q-1}{2}}+a)$ to points on an algebraic curve of degree $2$ and   the number of permutation binomials of the form $x^n(x^{\frac{q-1}{3}}+a)$ to rational points on an elliptic curve. Using this relationship, we calculate the exact number of permutation binomial of the form $x^n(x^{\frac{q-1}{2}}+a)$  (Theorem  \ref{item5}) and permutation binomial of the form $x^n(x^{\frac{q-1}{3}}+a)$ (Theorem \ref{item772}).
	\section{ Preliminaries}
	
	Throughout this article, $\F_q$ denotes the finite field with $q$ elements, where $q$ is a power of a prime $p$. 
For each non identically zero polynomial $f(x)\in \F_q[x]$, let define the index of $f$ as the smaller positive integer $m$ such such that $f$ can be write as $f(x)=x^r h(x^{(q-1)/m})+b$, where $b\in \F_q$ and $h(0)\ne 0$.   Many criteria for determining whether  a polynomial is a permutation polynomial depends to the numbers $r$ and $m$. In this line, we use extensively the following result, proved by Wan and Lidl  \cite[Theorem 1.2]{WaLi}. 
	\begin{theorem}\label{teoPerm}
		 Let $f(x):=x^r h(x^{\frac{q-1}{m}})+b\in \F_q[x]$ be a polynomial of index $m$ and $\alpha$ be a primitive element in $\Fq$. Then, $f(x)$ is a permutation polynomial of $\Fq$ if and only if the following conditions are satisfied:
		\begin{enumerate}[(i)]
			\item $\gcd\big(r, \frac{q-1}{m}\big)=1$;
			\item $h(\alpha^{i\frac{q-1}{m}})\neq 0,\text{ for all }i\text{ with }\ 0\leq i<m$;
			\item $f(\alpha^i)^{\frac{q-1}{m}}\neq f(\alpha^j)^{\frac{q-1}{m}},\text{ for all }i,j \text{ with }\ 0\leq i<j<m$.
		\end{enumerate}
	\end{theorem}
An essentially equivalent criterion for $f$ to permute $\F_q$ was given by Zieve \cite[Lemma 2.1]{Zie} and   Park and Lee \cite[Theorem 2.3]{PaLe}.
Note that, if the index of $f$ is small then this theorem gives an easy way to decide whether a polynomial permutes $\Fq$ and we explore this fact in the main results. However,  before we prove our principal results, we need of the following well known technical lemmas.

	\begin{lemma}\cite[Theorem 5.4]{LiNi}\label{item14} If $\chi$ is a nontrivial multiplicative character of $\Fq^*$, then
		\begin{center}
			$\sum\limits_{g\in\Fq^*}{\chi(g)} = 0.$
		\end{center}
	\end{lemma}

	\begin{lemma}\cite[Lemma 7.3]{LiNi}\label{item75}
		Let $m$ be a positive integer. We have
		
		$$\sum_{a\in\Fq}a^m=\begin{cases}
		0,\text{ if }(q-1)\nmid m\text{ or }m=0;\\
		-1\text{ if }(q-1)\mid m\text{ and }m\neq0,\\
		\end{cases}$$
		where $0^0:=1$.
	\end{lemma}

	From now, $\chi$ denote the non-trivial quadratic character of $\Fq^*$. 
	
	\begin{remark}\label{obs1}
		It is known that $\chi(a)=(-1)^n\in\C$ if and only if $a^{\frac{q-1}{2}}=(-1)^n\in\Fq$.
	\end{remark}

	The following lemma characterizes module $p$, the number of rational points of elliptic curves over $\Fp$.

	\begin{lemma}\label{item74}
		Let $E: y^2=x^3+Ax+B$ be  an elliptic curve over $\Fp$, where $p$ is a odd prime. The number of rational points of $E$ over $\Fp$ satisfies
the relation
		$$|E(\Fp)|-p-1\equiv -\sum_{l=\lceil\frac{p-1}{6}\rceil}^{\lfloor \frac{p-1}{4}\rfloor} \binom{\frac{p-1}{2}}{2l}\binom{2l}{\frac{p-1-2l}{2}}B^{\frac{p-1}{2}-2l}A^{3l-\frac{p-1}{2}}\!\!\!\!\pmod{p}$$
		
	\end{lemma}
	
	\begin{proof}
		We observe that
		
		$$1+\chi(x^3+Ax+B)=\begin{cases}
		0,\text{ if }x^3+Ax+B\text{ is not a square in }\Fp;\\
		2,\text{ if }x^3+Ax+B\text{ is a square in }\Fp.\\
		\end{cases}$$ 
		Therefore, considering $\infty$ as a solution, we have
		
		$$|E(\Fp)|=1+\sum_{x\in\Fp}[1+\chi(x^3+Ax+B)]=p+1+\sum_{x\in\Fp}\chi(x^3+Ax+B).$$
		By  Remark \ref{obs1},
		
		$$\begin{aligned}\sum_{x\in\Fp}\chi(x^3+Ax+B)
		&\equiv \sum_{x\in\Fp}(x^3+Ax+B)^{\frac{p-1}{2}}\\
		&\equiv \sum_{x\in\Fp}\sum_{i=0}^{\frac{p-1}{2}}\binom{\frac{p-1}{2}}{i}x^i(x^2+A)^iB^{\frac{p-1}{2}-i}\\
		&\equiv \sum_{x\in\Fp}\sum_{i=0}^{\frac{p-1}{2}}\binom{\frac{p-1}{2}}{i}B^{\frac{p-1}{2}-i}x^i\sum_{j=0}^{i}\binom{i}{j}x^{2j}A^{i-j}\\
		&\equiv \sum_{i=0}^{\frac{p-1}{2}}\sum_{j=0}^{i}\binom{\frac{p-1}{2}}{i}\binom{i}{j}A^{i-j}B^{\frac{p-1}{2}-i}\sum_{x\in\Fp}x^{2j+i}\pmod{p}.\\
		\end{aligned}$$
		
		By Lemma \ref{item75}, the  sum  $\sum_{x\in\Fp}x^{2j+i}$ is nonzero only if $2j+i\equiv 0\pmod{p-1}$ and $2j+i\neq 0$, and in these cases, the sum is $-1$. In addition, since  
		$$2j+i\le 2\cdot\frac{p-1}{2}+\frac{p-1}{2}=\frac{3(p-1)}{2}< 2(p-1),$$
it follows that 
		$$\begin{aligned}\sum_{i=0}^{\frac{p-1}{2}}\sum_{j=0}^{i}\binom{\frac{p-1}{2}}{i}\binom{i}{j}A^{i-j}B^{\frac{p-1}{2}-i}\sum_{x\in\Fp}x^{2j+i}
		&\equiv -\sum_{l=\lceil\frac{p-1}{6}\rceil}^{\lfloor \frac{p-1}{4}\rfloor} \binom{\frac{p-1}{2}}{2l}\binom{2l}{\frac{p-1-2l}{2}}B^{\frac{p-1}{2}-2l}A^{2i-\frac{p-1-2l}{2}}\\
		&\equiv -\sum_{l=\lceil\frac{p-1}{6}\rceil}^{\lfloor \frac{p-1}{4}\rfloor} \binom{\frac{p-1}{2}}{2l}\binom{2l}{\frac{p-1-2l}{2}}B^{\frac{p-1}{2}-2l}A^{3l-\frac{p-1}{2}}\pmod{p}.\\
		\end{aligned}$$ $\hfill\qed$
		
	\end{proof}

\begin{corollary}\label{item711}
	Let $E: y^2=x^3+B$ be  an elliptic curve over $\Fp$, where $p$ is a odd prime. The number of rational points of $E$ over $\Fp$ satisfies
the relation $|E(\Fp)|-p-1\equiv 0 \pmod{p}$  if  $p\equiv 2 \pmod 3$  and  
	$$|E(\Fp)|-p-1\equiv -\binom{\frac{p-1}{2}}{\frac{p-1}{3}} B^{\frac{p-1}{6}}\pmod{p}$$
if  $p\equiv 1 \pmod 3$.
\end{corollary}

	\begin{theorem}\cite[Theorem 2.3.1, Chapter V]{Sil}\label{item66} Let $E:y^2+a_1xy+a_2y=x^3+a_3x^2+a_4x+a_5$ be an elliptic curve over $\Fq$. Then, there is a complex number $\omega$ such that $|\omega|=\sqrt{q}$ and
		$$|E(\Fqk)|=q^k+1-\omega^k-\overline{\omega}^{\,k},$$
		for all positive integer $k$.
	\end{theorem}
	
	The motivation for the next two definitions comes from Lemma \ref{item71}, where we find the number of rational points of a specific elliptic curve.
	
	\begin{definition}\label{item70}
		Let $p\neq 3$ be a prime number. Let define 
$\kappa_7=1$ and $\kappa_{13}=-5$ and for $p\equiv 1\pmod{3}$,  $p\ge 19$, $\kappa_p$ is defined as  the unique integer that satisfies  $|\kappa_p|\leq 2\sqrt{p}$ and  
		$$\kappa_p\equiv-\binom{\frac{p-1}{2}}{\frac{p-1}{3}}\cdot4^{-\frac{p-1}{6}}\pmod{p}.$$
 In addition, for $p\equiv 2\pmod{3}$, let define $\kappa_p=0$.
	\end{definition}

From this definition, it is not clear that $\kappa_p$ exists for all $p$. But, as we will see, Lemma \ref{item71}  guarantee  the existence of $\kappa_p$. 	
	\begin{definition}\label{item76}
	Let $\pi_p$ denotes the complex number 	
		$$\frac{-\kappa_p}{2}+i\, \sqrt{p-\dfrac{\kappa_p^2}{4}}.$$
	\end{definition}

	\begin{lemma}\label{item71} Let $E:y^2=x^3+4^{-1}$ be an elliptic curve over $\Fp$, with $p$ prime and $p\not\in \{2,3\}$. Then 
		
		$$|E(\Fpk)|=p^k+1-\pi_p^k-\overline{\pi_p}^{\,k},$$
		for all positive integer $k$.
	\end{lemma}
	
	\begin{proof}
		By Theorem \ref{item66}, there exists  a complex number $\omega$ such that $|\omega|=\sqrt{q}$ for which
		
		$$|E(\Fp)|=p+1-\omega-\overline{\omega}.$$	
		By  Corollary  \ref{item711} and  Definition \ref{item76}, we conclude that $\omega=\pi_p$. Now, by Theorem \ref{item66}, we have	
		$$|E(\Fpk)|=p^k+1-\pi_p^k-\overline{\pi_p}^{\,k},$$
		for all positive integer $k$.$\hfill\qed$
	\end{proof}
	
	\section{Binomials of the form $x^n(x^{\frac{q-1}{2}}+a)$}
	
	In this section,  $\Fq$ denote a finite field of odd characteristic $p$.  Masuda-Zieve Theorem (Theorem \ref{teoMasuda}), in the case $r=2$, implies that the number $N$ of elements $a\in\Fq$ for which the binomial $x^n(x^{\frac{q-1}{2}}+a)$ permutes $\Fq$ satisfies 
	$$\frac{q-5}{2}\leq N\leq \frac{q+1}{2}.$$
One goal of this article is to determine the exact value of $N$ in this case  (Theorem \ref{item5}). In order to proof that theorem, we shall now show the necessary and sufficient conditions for that  binomial  of the form $x^n(x^{\frac{q-1}{2}}+a)$ to be a permutation polynomial. 

	\begin{lemma}\cite[Theorem 7.11]{LiNi}\label{item57}
		Let $n$ be a positive integer and $a$ an element in $\Fq^*$. Then, $f(x):=x^n(x^{\frac{q-1}{2}}+a)$ is a permutation polynomial if and only if
		
		\begin{enumerate}[(a)]
			\item $\gcd\big(n,\tfrac{q-1}{2}\big)=1$;\vskip0.1cm
			\item $\chi(a^2-1)=(-1)^{n+1}$.
		\end{enumerate}
	\end{lemma}
	
	\begin{proof}
		It is enough to show that the condition $(b)$ is equivalent to the condition $(iii)$ in Theorem \ref{teoPerm}. Let $\alpha$ be an primitive element in $\Fq$ and observe that 
		$$\alpha^{\frac{n(q-1)}{2}}(\alpha^{\frac{q-1}{2}}+a)^{\frac{q-1}{2}}\neq \alpha^{n\frac{2(q-1)}{2}}(\alpha^{\frac{2(q-1)}{2}}+a)^{\frac{q-1}{2}}$$
		is equivalent to $(-1)^n(a-1)^{\frac{q-1}{2}}\neq 1\cdot(a+1)^{\frac{q-1}{2}},$ which is the same as $(a^2-1)^{\frac{q-1}{1}}\neq(-1)^n$. So, the result follows from Remark \ref{obs1}.$\hfill\qed$
	\end{proof}
	
	\begin{theorem}\label{item5}
		Let $n$ be an integer such that $\gcd(n,\frac{q-1}{2})=1$. The number of elements $a\in\Fq$ for which the binomial $x^n(x^{\frac{q-1}{2}}+a)$ permutes $\Fq$ is given by the formula
		$$\frac{q-2+(-1)^n}{2}.$$
	\end{theorem}
	\begin{proof}  
	Let $M$ denote the number of elements $a\in\Fq$ for which  $\chi(a^2-1)=1$, and in this case,  if $b:=a^2\ne 1$, then $b$ and $b-1$ are  simultaneously squares in $\Fq$.	In other hand,	
		$$\big(1+\chi(b)\big)\big(1+\chi(b-1)\big)=\begin{cases}
		0,\text{ if }\text{either } b\text{ or }b-1\text{ is not a square residue;}\\
		4,\text{ if }b\text{ and }b-1\text{ are square residues and } b(b-1)\ne 0\\
		2,\text{ if }b\text{ and }b-1\text{ are square residues and } b(b-1)= 0\\
		\end{cases}$$\
Using these equalities we have
		$$2M=\sum_{\substack{b\in\Fq\\ b\neq1}}\big(1+\chi(b)\big)\big(1+\chi(b-1)\big)=\sum_{\substack{b\in\Fq\\ b\neq1}}1+\sum_{\substack{b\in\Fq\\ b\neq1}}\chi(b)+\sum_{\substack{b\in\Fq\\ b\neq1}}\chi(b-1)+\sum_{\substack{b\in\Fq\\ b\neq1}}\chi(b(b-1)).$$
The first summation is equal to $q-1$ and, by Lemma \ref{item14}, we conclude that the second summation is equal to $-1$ and the third summation is equal to $0$. In order to calculate the last summation, by Remark \ref{obs1}, we have	
		$$\sum_{\substack{b\in\Fq\\ b\neq1}}\chi(b(b-1))=\sum_{\substack{b\in\Fq\\ b\neq1}}\chi(b^2)\chi(1-b^{-1})=\sum_{\substack{x\in\Fq\\ x\neq1}}\chi(x)=-\chi(1)=-1.$$
		Therefore, we conclude that $M=\frac{q-3}{2}$. 
		
		Now, by Theorem \ref{item57} the number of $a\in\Fq$ for which the binomial $x^n(x^{\frac{q-1}{2}}+a)$ permutes $\Fq$ is given by $M=\frac{q-3}{2}$ if $n$ is odd and $|\Fq\backslash\{\pm 1\}|-M=\frac{q-1}{2}$ if $n$ is even.
%
%
%
%
$\hfill\qed$
		
	\end{proof}

	\section{Binomials of the form $x^n(x^{\frac{q-1}{3}}+a)$}
	
	In this section, we assume that $q\equiv 1\pmod{3}$, $\xi\ne1$ be a cubic root of unit in $\Fq$ and $\delta\ne 1$ be a cubic root of unit in $\C$. Let  $\eta$ be the cubic multiplicative character that satisfies
	$$\eta(a)=\delta^n\text{ if and only if }a^{\frac{q-1}{3}}=\xi^n,\text{ for all }a\in\Fq^*$$
and extend $\eta$ to $\Fq$  defining $\eta(0)=0$. By Masuda-Zieve Theorem (Theorem \ref{teoMasuda}), the number $N$ of elements $a\in\Fq$ for which the binomial $x^n(x^{\frac{q-1}{3}}+a)$ permutes $\Fq$ satisfies 
	$$\frac{2}{9}\big(q-35-20\sqrt{q}\big)\leq N\leq \frac{2}{9}\big(q+1+20\sqrt{q}\big).$$\
	In Theorem \ref{item772}, we determine the exact value of $N$. In order to proof the principal result of this section, we need the following technical lemmas. 
	
	\begin{lemma}\label{item59}
		Let $n$ be a positive integer. The polynomial $f(x):=x^n(x^{\frac{q-1}{3}}+a)\in\Fx$ is a permutation polynomial over $\Fq$ if and only if the following conditions are satisfied
		\begin{enumerate}[(a)]
			\item $\gcd\big(n,\tfrac{q-1}{3}\big)=1$;\vskip0.15cm
			\item $a \not\in\{-1,-\xi,-\xi^2\}$;\vskip0.15cm
			\item $\eta\big(\frac{\xi+a}{1+a}\big)\neq\delta^{2n}$;\vskip0.15cm
			\item $\eta\big(\frac{1+a}{\xi^2+a}\big)\neq\delta^{2n}$;\vskip0.15cm
			\item $\eta\big(\frac{\xi^2+a}{\xi+a}\big)\neq\delta^{2n}$.
		\end{enumerate}
	\end{lemma}

	\begin{proof}
		Observe that the condition $(b)$ is equivalent to condition $(ii)$ in Theorem \ref{teoPerm}. Thus, it is sufficient to show that the conditions $(c)$, $(d)$ and $(e)$ are equivalent to the condition $(iii)$ in Theorem \ref{teoPerm}.  Since $\xi=\alpha^{(q-1)/3}$ for some  $\alpha$  primitive element in $\Fq$  that  condition is the same as  
		$$\xi^{ni}(\xi^i+a)^{\frac{q-1}{3}}\neq \xi^{nj}(\xi^j+a)^{\frac{q-1}{3}},\text{ with }0\leq i<j<3,$$
	and we can rewrite  these inequalities as  
		$$\dfrac{\xi^{n}(\xi+a)^{\frac{q-1}{3}}}{\xi^{0\cdot n}(1+a)^{\frac{q-1}{3}}}\neq 1,\  \dfrac{\xi^{n}(1+a)^{\frac{q-1}{3}}}{\xi^{2n}(\xi^2+a)^{\frac{q-1}{3}}}\neq1\ \text{ and }\ \dfrac{\xi^{2n}(\xi^2+a)^{\frac{q-1}{3}}}{\xi^{n}(\xi+a)^{\frac{q-1}{3}}}\neq1.$$
	By the choice of the character $\eta$, these  are equivalent to
		$$\eta\bigg(\frac{\xi+a}{1+a}\bigg)\neq\delta^{2n},\ \eta\bigg(\frac{1+a}{\xi^2+a}\bigg)\neq\delta^{2n}\ \text{ and }\ \eta\bigg(\frac{\xi^2+a}{\xi+a}\bigg)\neq\delta^{2n}.$$
			Hence, the result follows. $\hfill\qed$
	\end{proof}


	\begin{lemma}\label{item65}
		Let $\Lambda:=\Fq\backslash\{-1,-\xi,-\xi^2\}$. Then
		
		$$\delta^n\sum\limits_{a\in\Lambda}\eta\bigg(\frac{\xi+a}{1+a}\bigg)+\delta^{2n}\sum\limits_{a\in\Lambda}\eta^2\bigg(\frac{\xi+a}{1+a}\bigg)=\epsilon_1+\epsilon_2,$$
		where
		
		$$\epsilon_1=\begin{cases}
		-2,\text{ if }q-3n\equiv 1\pmod{9};\\
		1,\text{ if }q-3n\not\equiv 1\pmod{9}.\\
		\end{cases}\text{ and }\ \ \epsilon_2=\begin{cases}
		-2,\text{ if }n\equiv0\pmod{3};\\
		1,\text{ if }n\not\equiv0\pmod{3}.\\
		\end{cases}$$
	\end{lemma}

	\begin{proof}
		We observe that the function $f:\Fq\backslash\{-1\}\rightarrow\Fq\backslash\{1\}$ defined by $x\mapsto\frac{\xi+x}{1+x}$ is a bijective function. Since $f(x)$ is injective, 
		
		$$\delta^n\sum\limits_{a\in\Lambda}\eta\bigg(\frac{\xi+a}{1+a}\bigg)+\delta^{2n}\sum\limits_{a\in\Lambda}\eta^2\bigg(\frac{\xi+a}{1+a}\bigg)=\delta^n\sum\limits_{b\in \Im(f|_\Lambda)}\eta(b)+\delta^{2n}\sum\limits_{b\in \Im(f|_\Lambda)}\eta^2(b),$$
where $\Im(f|_\Lambda)$ denotes the image of the function $f$ considering the domain $\Lambda$.
		
It follows from Lemma \ref{item14} that
		$$\delta^n\sum\limits_{b\in \Im(\Lambda)}\eta(b)+\delta^{2n}\sum\limits_{b\in \Im(\Lambda)}\eta^2(b)=-\delta^n\sum\limits_{b\in \Fq\backslash\Im(f|_\Lambda)}\eta(b)-\delta^{2n}\sum\limits_{b\in \Fq\backslash\Im(f|_\Lambda)}\eta^2(b).$$\\
		
Since $f(x)$ is a bijective function, $\Fq\backslash\Im(f|_\Lambda)=\{1,f(-\xi),f(-\xi^2)\}=\{1,0,-\xi^2\}$. Then,
		
		$$\begin{aligned}
		-\delta^n\sum\limits_{b\in \Fq\backslash\Im(f|_\Lambda)}\eta(b)-\delta^{2n}\sum\limits_{b\in \Fq\backslash\Im(f|_\Lambda)}\eta^2(b)
		&=\delta^n[-\eta(1)-\eta(-\xi^2)]+\delta^{2n}[-\eta^2(1)-\eta^2(-\xi^2)]\\
		&=-\delta^n-\delta^{2n}-\big(\delta^{-n}\delta^{\frac{q-1}{3}}\big)^2-\big(\delta^{-n}\delta^{\frac{q-1}{3}}\big).\\
		\end{aligned}$$
		Finally,  defining $\epsilon_1=-\delta^n-\delta^{2n}$ and $\epsilon_2=-\big(\delta^{-n}\delta^{\frac{q-1}{3}}\big)^2-\big(\delta^{-n}\delta^{\frac{q-1}{3}}\big)$ and using the fact that 
		$$1+\delta^l+\delta^{2l}=\begin{cases} 0 &\text{if } l\not\equiv 0\pmod 3\\
		 3&\text{if } l\equiv 0\pmod 3,
		 \end{cases}$$ we obtain the result. $\hfill\qed$
	\end{proof}

	\begin{lemma}\label{item67}
		Let $\Fq$ be a finite field with $q=p^k$ elements, where $p$ is an odd prime. Let denote $\Lambda=\Fq\backslash\{-1,-\xi,-\xi^2\}$. Then
		$$\sum\limits_{a\in\Lambda}\eta\bigg(\frac{a^2-a+1}{a^2+2a+1}\bigg)+\eta^2\bigg(\frac{a^2-a+1}{a^2+2a+1}\bigg)=-2-\pi_p^{\,k}-\overline{\pi_p}^{\,k},$$
		where $\pi_p$ is as in Definition \ref{item76}.
	\end{lemma}

	\begin{proof}
		Observe that
		$$\eta(b)+\eta^2(b)=\begin{cases}
		2,\text { if there is }c\in\Fq^*\text{ such that }b=c^3;\\
		-1,\text {in otherwise}.\\
		\end{cases}$$
		Let 
		$S:= \sum\limits_{a\in\Lambda}\eta\big(\frac{a^2-a+1}{a^2+2a+1}\big)+\eta^2\big(\frac{a^2-a+1}{a^2+2a+1}\big)$ and $f(x):=\tfrac{x^2-x+1}{x^2+2x+1}$. Thus,
		\begin{equation}\label{equa52}
		S=2\cdot\big|\big\{a\in\Lambda:f(a)\text{ is a cube in }\Fq\big\}\big|-1\cdot\big|\big\{a\in\Lambda:f(a)\text{ is not a cube in }\Fq\big\}\big|.
		\end{equation}
		
		Therefore, it is enough to calculate the cardinality of $\mathcal{A}:=\{a\in\Lambda:f(a)\text{ is a cube in }\Fq\}$. Given $b$ an element in $\Fq^*$, we want to determine  the elements $a\in\Lambda$ for which $f(a)=b^3$ for some $b\in \Fq$. However, assuming that $a$ is an element in $\Lambda$ such that $f(a)=b^3$, we get
		
		\begin{equation}\label{equa50}
		a^2(1-b^3)-a(1+2b^3)+1-b^3=0.
		\end{equation}
		
		If $b^3=1$, then $a=0$. If $b^3\neq 1$, this equation has two solutions in $\mathbb{F}_{q^2}$, given by
		\begin{equation}\label{equa51}
		a=\frac{1+2b^3\pm\sqrt{\Delta}}{2(1-b^3)},
		\end{equation}\\
		where $\Delta:=(1+2b^3)^2-4(1-b^3)^2=(1+2b^3+2-2b^3)(1+2b^3-2+2b^3)=-3+12b^3$. Note that the values for $a$ given by \eqref{equa51} are in $\Fq$ if and only if $\Delta=-3+12b^3$ is a square in $\Fq$. Hence, fixed $b\in\Fq$, there exists $a\in\Lambda$ such that $f(a)=b^3$ if and only if there exists $x_a\in\Fq$ such that $x_a^2=-3+12b^3$. Thus, this last problem is equivalent to finding the number of rational points of the curve 
		
		$$E:x^2=-3+12y^3$$
		over $\Fq$. It should be noted that some rational points in $E(\Fq)$ are related to elements that are not in $\Lambda$, which are $(\xi-\xi^2,0)$ and $(\xi^2-\xi,0)$. Besides these, the six rational points $(\pm3,1)$, $(\pm3,\xi)$, $(\pm3,\xi^2)\in E(\Fq)$ should also not be considered, because are related to $b=1$. Now, let $G=\{\infty, (\xi-\xi^2,0), (\xi^2-\xi,0), (\pm3,1), (\pm3,\xi), (\pm3,\xi^2)\}\subset E(\Fq)$ and let
		
		$$\varphi:E(\Fq)\backslash G\longrightarrow \mathcal{A}\backslash\{0\}$$\\
		be a function defined by 
		
		$$(x,y)\mapsto \frac{1+2y^3+x}{2(1-y^3)}.$$
		
		Clearly, by definition of $\mathcal{A}$, $\varphi$ is a surjective function. In addition, if $(x_0,y_0)$ is a rational point in $E(\Fq)\backslash G$, then $(x_0,\xi y_0),(x_0,\xi^2y_0)\in E(\Fq)\backslash G$ are such that $\varphi((x_0,y_0))=\varphi((x_0,\xi y_0))=\varphi((x_0,\xi^2 y_0))$. We observe also that if $(x_1,y_1)$ is rational point in $E(\Fq)\backslash G$ such that $(x_1,y_1)$ does not belongs to $\{(x_0,y_0),(x_0,\xi y_0),(x_0,\xi^2y_0)\}$, then $\varphi((x_1,y_1))\neq \varphi((x_0,y_0))$. Consequently, 
		
		\begin{equation}\label{equa10}
		|E(\Fq)\backslash G|=3\cdot|\mathcal{A}\backslash\{0\}|.
		\end{equation}
		
		Thus, in order to calculate $|\mathcal{A}|$, we only need to find the number of elements in $E(\Fq)$. Now, if we make the change of variables $(x,y)\mapsto((2\xi^2-2\xi)x,-y)$ in the curve $E$, then we get $E':x^2=4^{-1}+y^3$. Then $|E(\Fq)|=|E'(\Fq)|$. By Lemma \ref{item71} and by Equation \eqref{equa10}, we have
		
		$$|\mathcal{A}|=|\mathcal{A}\backslash\{0\}|+1=\frac{|E(\Fq)|-|G|}{3}+1=\frac{q-5-\pi_p^k-\overline{\pi_p}^{\,k}}{3}.$$
		
		Hence, by Equation \eqref{equa52}, the result is proved.$\hfill\qed$
	\end{proof}

	 \vskip 0.4cm 
	The following lemma shows an equivalent result of Lemma \ref{item67} but in fields with $4^k$ elements.
	
	\begin{lemma}\label{item78}
		Let $\Lambda:=\mathbb{F}_{4^k}\backslash\{-1,-\xi,-\xi^2\}$. Then
		$$S:=\sum\limits_{a\in\Lambda}\eta\bigg(\frac{a^2+a+1}{a^2+1}\bigg)+\eta^2\bigg(\frac{a^2+a+1}{a^2+1}\bigg)=-2+(-2)^{k+1}.$$
	\end{lemma}

	\begin{proof}
		Let $f(x):=\tfrac{x^2+x+1}{x^2+1}$. By the same argument used to  prove  the previous lemma, we have
		
		\begin{equation}\label{equa55}
		S=2\cdot\big|\big\{a\in\Lambda:f(a)\text{ is a cube in }\mathbb{F}_{4^k}\big\}\big|-1\cdot\big|\big\{a\in\Lambda:f(a)\text{ is not a cube in }\mathbb{F}_{4^k}\big\}\big|.
		\end{equation}
		
		Therefore, it is enough to calculate the cardinality of $\mathcal{A}:=\{a\in\Lambda:f(a)\text{ is a cube in }\mathbb{F}_{4^k}\}$. Given $b$ an element in $\mathbb{F}_{4^k}^*$, we want to find what are the elements $a\in\Lambda$ for which $f(a)=b^3$. However, assuming that $a$ is an element in $\Lambda$ such that $f(a)=b^3$, we get
		
		\begin{equation}\label{equa56}
		a^2(1-b^3)-a+1-b^3=a^2(1+b^3)+a+1+b^3=0.
		\end{equation}
		
		If $b^3=1$, then $a=0$. Now, we consider $b^3\neq 1$. Making the change of variables $x=a(1+b^3)$ and $y=b^{2^{2k-1}}$, we get the equation 
		
		\begin{equation}\label{equa57}
		x^2+x+1+y^3=0.
		\end{equation}
		
		Since $(a,b)\mapsto(a(1+b^3),b^{2^{2k-1}})$ is a bijective function in $\Fq\times\Fq\backslash\{1,\xi,\xi^2\}$, the number of rational points on \eqref{equa56} and \eqref{equa57} is the same. Now, let $E:x^2+x=y^3+1$ be a elliptic curve over $\mathbb{F}_{4^k}$. We let $G=\{\infty, (\xi,0), (\xi^2,0), (0,1), (0,\xi), (0,\xi^2), (1,1), (1,\xi), (1,\xi^2)\}\subset E(\mathbb{F}_{4^k})$ and let
		
		$$\varphi:E(\mathbb{F}_{4^k})\backslash G\longrightarrow \mathcal{A}\backslash\{0\}$$\\
		defined by 
		
		$$(x,y)\mapsto x(1+y^6)^{-1}.$$
		As in the previous lemma, it follows that 
		
		\begin{equation}\label{equa14}
		|E(\mathbb{F}_{4^k})\backslash G|=3\cdot|\mathcal{A}\backslash\{0\}|,
		\end{equation}
		and by Theorem \ref{item66}, there exists a complex number $\omega\in\C$, with $|\omega|=\sqrt{2}$, such that
		
		\begin{equation}\label{equa16}
		|E(\mathbb{F}_{2^m})|=2^m+1-\omega^m-\overline{\omega}^{\,m},\text { for all  }m\in\N.
		\end{equation}
		The rational points of $E:x^2+x=y^3+1$ over $\F_2$ are $\infty, (0,1)$ and $(1,1)$, then 
		
		$$2^n+1-\omega-\overline{\omega}=|E(\mathbb{F}_{2})|=3,$$
		thus $\omega=\sqrt{2}i$. Finally, by Equation \eqref{equa14} and Equation \eqref{equa16}, we get
		
		$$|\mathcal{A}|=|\mathcal{A}\backslash\{0\}|+1=\frac{|E(\mathbb{F}_{4^k})|-|G|}{3}+1=\frac{4^k-5+(-2)^{k+1}}{3}.$$
		Hence, by Equation \eqref{equa55}, the result is proved.$\hfill\qed$\\
	\end{proof}
	
	Our main result is the following.
	
	\begin{theorem}\label{item772}
		Let $\Fq$ be a finite field with characteristic $p$ and $q=p^k$. Assume $q\equiv 1\pmod{3}$. Let $n$ be a positive integer such that $\gcd(n,\frac{q-1}{3})=1$. The number of elements $a\in\Fq$ for which the binomial $f(x)=x^n(x^{\frac{q-1}{3}}+a)$ permutes $\Fq$ is given by 
		
		$$\frac{2q-3(\epsilon_1+\epsilon_2)-10-2(\pi_p^{\,k}+\overline{\pi_p}^{\,k})}{9},$$
		where $\pi_p$ is as in Definition \ref{item76} and
		$$\epsilon_1=\begin{cases}
		-2,\text{ if }q-3n\equiv 1\pmod{9};\\
		1,\text{ if }q-3n\not\equiv 1\pmod{9}.\\
		\end{cases}\text{ and }\ \ \epsilon_2=\begin{cases}
		-2,\text{ if }n\equiv0\pmod{3};\\
		1,\text{ if }n\not\equiv0\pmod{3}.\\
		\end{cases}$$
	\end{theorem}

	\begin{proof}Let $\Lambda:=\Fq\backslash\{-1,-\xi,-\xi^2\}$. By Lemma \ref{item59}, it is enough to calculate the number of elements $a\in\Lambda$ for which 
		$$\eta\left(\frac{\xi+a}{1+a}\right)\neq\delta^{2n}, \eta\left(\frac{1+a}{\xi^2+a}\right)\neq\delta^{2n}\text{ and }\eta\left(\frac{\xi^2+a}{\xi+a}\right)\neq\delta^{2n},$$
		Let denote by $N$ the number of such elements. To simplify the notation, let 
		
		\begin{equation}\label{equa48}
		\lambda_0:=\frac{\xi+a}{1+a}, \lambda_1:=\frac{1+a}{\xi^2+a}\text{ and }\lambda_2:=\frac{\xi^2+a}{\xi+a},
		\end{equation}
		leaving implicit the dependence on $a$. Note that
		
		\begin{equation}\label{equa60}
		2-\frac{\eta(x)}{\delta^{2n}}-\frac{\eta^2(x)}{\delta^{4n}}=\begin{cases}
		0,\text{ if }\eta(x)=\delta^{2n};\\
		3,\text{ if }\eta(x)\neq\delta^{2n},\\
		\end{cases}
		\end{equation}
Thus we know that
\begin{equation}\label{equa61}
N=\frac{1}{27}\sum\limits_{a\in\Lambda}\bigg[2-\frac{\eta(\lambda_0)}{\delta^{2n}}-\frac{\eta^2(\lambda_0)}{\delta^{n}}\bigg]\bigg[2-\frac{\eta(\lambda_1)}{\delta^{2n}}-\frac{\eta^2(\lambda_1)}{\delta^{n}}\bigg]\bigg[2-\frac{\eta(\lambda_2)}{\delta^{2n}}-\frac{\eta^2(\lambda_2)}{\delta^{n}}\bigg]
\end{equation}
In order to calculate $N$, we  
		 use the fact that 
		$$\{\eta(\lambda_0)\}_{a\in\Lambda}=\{\eta(\lambda_1)\}_{a\in\Lambda}=\{\eta(\lambda_2)\}_{a\in\Lambda},$$
so, rewriting Equation \eqref{equa61} as $N=\frac 1{27} \sum_{j=1}^7 N_j$, where\\
	
	
		$\begin{aligned}		
		N_1&=\sum\limits_{a\in\Lambda}\bigl( 8-\eta(\lambda_0)\eta(\lambda_1)\eta(\lambda_2)-
		\eta^2(\lambda_0)\eta^2(\lambda_1)\eta^2(\lambda_2)\bigr)\\
		&=\sum\limits_{a\in\Lambda}\bigl(8-\eta(1)-\eta^2(1)\bigr)\\
		&=6|\Lambda|=6(q-3).\\
		\end{aligned}$\\[0.8em]
		
		$\begin{aligned}
		N_2&=-4\delta^{n}\sum\limits_{a\in\Lambda}\bigl(\eta(\lambda_0)+\eta(\lambda_1)+\eta(\lambda_2)\bigr)
		-4\delta^{2n}\sum\limits_{a\in\Lambda}\bigl(\eta^2(\lambda_0)+\eta^2(\lambda_1)+\eta^2(\lambda_2)\bigr)\\
	&=-12\delta^{n}\sum\limits_{a\in\Lambda}\eta(\lambda_0)-12\delta^{2n}\sum\limits_{a\in\Lambda}\eta^2(\lambda_0).
		\end{aligned}$\\[0.8em]
		
		$\begin{aligned}
N_3&=-\delta^{n}\sum\limits_{a\in\Lambda}
\bigl(\eta^2(\lambda_0)\eta(\lambda_1)\eta(\lambda_2)+
\eta^2(\lambda_1)\eta(\lambda_0)\eta(\lambda_2)+\eta^2(\lambda_2)\eta(\lambda_0)\eta(\lambda_1)\bigr)\\
&=-\delta^{n}\sum\limits_{a\in\Lambda}
\bigl(\eta(\lambda_0)\eta(1)+\eta(\lambda_1)\eta(1)+\eta(\lambda_2)\eta(1)\bigr)\\
		&=-3\delta^{n}\sum\limits_{a\in\Lambda}\eta(\lambda_0).\\
		\end{aligned}$\\[0.8em]
		
		$\begin{aligned}	N_4&=-\delta^{2n}\sum\limits_{a\in\Lambda}
\bigl(\eta(\lambda_0)\eta^2(\lambda_1)\eta^2(\lambda_2)
+\eta(\lambda_1)\eta^2(\lambda_0)\eta^2(\lambda_2)+\eta(\lambda_2)\eta^2(\lambda_0)\eta^2(\lambda_1)\bigr)\\	&=-\delta^{2n}\sum\limits_{a\in\Lambda}
\bigl(\eta^2(\lambda_0)\eta^2(1)+\eta^2(\lambda_1)\eta^2(1)+\eta^2(\lambda_2)\eta^2(1)\bigr)\\
		&=-3\delta^{2n}\sum\limits_{a\in\Lambda}\eta^2(\lambda_0).\\
		\end{aligned}$\\[0.8em]
		
		$\begin{aligned}
		N_5&=2\delta^{2n}\sum\limits_{a\in\Lambda}
		\bigl(\eta(\lambda_1)\eta(\lambda_2)+\eta(\lambda_0)\eta(\lambda_3)+\eta(\lambda_1)\eta(\lambda_3)\bigr)\\
		&=2\delta^{2n}\sum\limits_{a\in\Lambda}\bigl(\eta(\lambda_0^{-1})
		+\eta(\lambda_1^{-1})+\eta(\lambda_2^{-1})\bigr)\\
		&=6\delta^{2n}\sum\limits_{a\in\Lambda}\eta^2(\lambda_0).\\
		\end{aligned}$\\[0.8em]
		
		$\begin{aligned}
		N_6&=2\delta^{n}\sum\limits_{a\in\Lambda}\bigl(\eta^2(\lambda_1)\eta^2(\lambda_2)
		+\eta^2(\lambda_0)\eta^2(\lambda_3)+\eta^2(\lambda_1)\eta^2(\lambda_3)\bigr)\\
		&=2\delta^{n}\sum\limits_{a\in\Lambda}
		\bigl(\eta^2(\lambda_0^{-1})+\eta^2(\lambda_1^{-1})+\eta^2(\lambda_2^{-1})\bigr)\\
		&=6\delta^{n}\sum\limits_{a\in\Lambda}\eta(\lambda_0).\\
		\end{aligned}$\\[0.8em]
		
		$\begin{aligned}
		N_7&=2\sum\limits_{a\in\Lambda}
		\bigl(\eta(\lambda_0\lambda_1^{-1})+\eta(\lambda_0\lambda_2^{-1})+\eta(\lambda_1\lambda_0^{-1})
		+\eta(\lambda_1\lambda_2^{-1})+\eta(\lambda_2\lambda_0^{-1})+\eta(\lambda_2\lambda_1^{-1})\bigr)\\
		&=6\sum\limits_{a\in\Lambda}\bigl(\eta(\lambda_0\lambda_1^{-1})+\eta(\lambda_1\lambda_0^{-1})\bigr)\\
		&=6\sum\limits_{a\in\Lambda}\bigl(\eta(\lambda_0\lambda_1^{-1})+\eta^2(\lambda_0\lambda_1^{-1})\bigr).\\
		\end{aligned}$
		
		Hence, 
		$$N\!=\frac{1}{27}\sum\limits_{j=1}^{7}N_j\!=\frac{6(q-3)}{27}-\frac{9}{27}\Bigg[\delta^n\sum\limits_{a\in\Lambda}\eta(\lambda_0)+\delta^{2n}\sum\limits_{a\in\Lambda}\eta^2(\lambda_0)\Bigg]\!+\frac{6}{27}\sum\limits_{a\in\Lambda}
		\bigl(\eta(\lambda_0\lambda_1^{-1})+\eta^2(\lambda_0\lambda_1^{-1})\bigr).$$
		By Lemma \ref{item65}, we have 
		$$\delta^n\sum\limits_{a\in\Lambda}\eta(\lambda_0)+\delta^{2n}\sum\limits_{a\in\Lambda}\eta^2(\lambda_0)=\epsilon_1+\epsilon_2.$$	
		and by Lemma \ref{item67} and Lemma \ref{item78}, we obtain	
		$$\sum\limits_{a\in\Lambda}\bigl(\eta(\lambda_0\lambda_1^{-1})+\eta^2(\lambda_0\lambda_1^{-1})\bigr)
		=-2-\pi_p^{\,k}-\overline{\pi_p}^{\,k}.$$
		Using the last three equations we obtain the result.$\hfill\qed$
	\end{proof}

\begin{corollary}\label{asinto3}
Let $p$ be a prime number, $k$ be a positive integer and $\Fq$ be a finite field with $q=p^k$ elements. Assume $q\equiv 1\pmod{3}$. Let $n$ be a positive integer such that $\gcd(n,\frac{q-1}{3})=1$ and $T_k$ be the number of elements $a\in\Fq$ for which the binomial $f(x)=x^n(x^{\frac{q-1}{3}}+a)$ permutes $\Fq$. Then
\begin{equation}\label{caso3}
\left\lceil\frac{2q-4\sqrt q-16}9\right\rceil \le T_k\le \left\lfloor\frac{2q+4\sqrt q-7} 9\right\rfloor.
\end{equation}
In addition, these inequalities are asymptotically sharp, i.e.,  for every $\epsilon>0$ there exist infinite values of $k$ and $k'$ such that 
\begin{equation}\label{caso3_1}
T_k<\frac 29q- \left(\frac 49-\epsilon\right)\sqrt q\ \text{ and }\ T_{k'}>\frac 29q +\left(\frac 49-\epsilon\right)\sqrt q.
\end{equation}
\end{corollary}

\begin{proof}
Since  $\gcd(n,\frac{q-1}{3})=1$, by the definition of $\epsilon_1$ and $\epsilon_2$, it is not possible that $\epsilon_1=\epsilon_2=-2$, then $-1\le \epsilon_1+\epsilon_2\le 2$. So, the inequalities in (\ref{caso3}) follow from $|\pi_p^k|=p^{k/2}=\sqrt q$. 

In order to prove the inequalities in (\ref{caso3_1}), we observe that, in the case that the characteristic is congruent to $2\pmod 3$, then  $\pi_p=i\sqrt p$ and the bounds are achieved alternately  for every even number $k$.
Thus, it remains to consider the case when the characteristic is congruent to $1 \pmod 3$. Since
$\pi_p=\sqrt p e^{i\theta_p}$,  where  $\theta_p=\arctan\left(\frac{\sqrt{4p-k_p^2}}{k_p}\right),$
then $\pi_p^k+\overline{\pi_p}^k= 2 \sqrt q \cos (k\theta_p)$. 

If $\frac {\theta_p}{\pi}=\frac rs$ is a rational number,  and since $\tan(\frac{r\pi}s)$ has algebraic degree $\le 2$ over $\Q$, then it follows from Theorem 3.11 in \cite{Niv} that $s\in \{1,3,4,6,8,12\}$.
However, it is easy to verify that  the equation $\frac{\sqrt{4p-a^2}}a=\tan(\frac{\pi r}s)$ with $a\in\Z$, $s\in \{1,3,4,6,8,12\}$,  $1\le r<s$, $\gcd(r,s)=1$ and $p\equiv 1\pmod 3$ does not have integer solutions.  So, it follows that  $\frac {\theta_p}{\pi}$ is an irrational number.

Now, considering  any convergent $\frac{m_l}{n_l}$ of the continued  fraction of  $\frac {\theta_p}{2\pi}$, we know that it  satisfies the inequality $|\frac {\theta_p}{2\pi}-\frac{m_l}{n_l}|<\frac 1{ n_l^2}$ (see \cite[Theorem 164]{HWHSW}),  then $n_l\theta_p=2\pi m_l+ \frac {\delta_l}{n_l}$ with $|\delta_l|< 2\pi$. Therefore
$$\pi_p^k+\overline{\pi_p}^k= 2 \sqrt q \cos \left(n_l\theta_p\right)=2\sqrt q\cos\left(\frac{\delta_l}{n_l}\right)> \sqrt q\left(2-\frac{\delta_l^2}{n_l^2}\right). $$

In the same way, considering the  convergent $\frac{m_l}{n_l}$ of the continued fraction of $\frac {\theta_p}{\pi}$,  and using the fact that $|m_ln_{l+1}-m_{l+1}n_l|=1$ (see \cite[Theorem 150]{HWHSW}),  thus having at least one of the  numbers $m_l$, $m_{l+1}$  odd, we have that  there exist infinitely many integers $l$ such that $m_l$ is odd and
 $$\pi_p^k+\overline{\pi_p}^k=2 \sqrt q \cos \left(n_l\theta_p\right)=2\sqrt q\cos\left(m_l\pi+ \frac{\delta_l}{n_l}\right)=-2\sqrt q\cos\left(\frac{\delta_l}{n_l}\right)< -\sqrt q\left(2-\frac{\delta_l^2}{n_l^2}\right),$$
where $|\delta_l|<\pi$.  These last two inequalities and Theorem \ref{item772} implied the inequalities in (\ref{caso3_1})  and then  the result  in the Equation (\ref{caso3}) is asymptotically sharp.$\hfill\qed$
\end{proof}

		\begin{example}Let $p=73$,  $\Fq$ be a finite field with $p^k$ elements and $n=35$.
		By definition, $\kappa_{73}=7$ and $\pi_{73}=\frac{-7+i \sqrt{243}}{2}$. Thus, if $\gcd(35,\frac{73^k-1}{3})=1$, then Theorem \ref{item772} states that the number of elements $a\in\F_{73^k}$ for which the polynomial $x^{35}(x^{\frac{73^k-1}{3}}+a)$ permutes $\F_{73^k}$ is given by
		
		$$T_k=\frac{2\cdot 73^k-16-2(\frac{-7+i \sqrt{243}}{2})^{\,k}-2(\frac{-7-i \sqrt{243}}{2})^{\,k}}{9}.$$
		
		In particular, $T_1=16$. In fact, using SageMath  software (\cite{sage}) is easy to verify that 
		$x^{59}$, $x^{35}(x^{24}+2)$, $x^{35}(x^{24}+4)$, $x^{35}(x^{24}+16)$, $x^{35}(x^{24}+18)$, $x^{35}(x^{24}+21)$, $x^{35}(x^{24}+22),$
		$ x^{35}(x^{24}+30)$, $x^{35}(x^{24}+32)$,  $x^{35}(x^{24}+33)$, $x^{35}(x^{24}+37)$, $x^{35}(x^{24}+45)$, $x^{35}(x^{24}+55)$, $x^{35}(x^{24}+57)$,
		$ x^{35}(x^{24}+68)$ and  $x^{35}(x^{24}+71)$ are all the permutation binomials of that form.

In addition,  for $k_1=1217$ and $k_2=1578$  we have that
$$ \frac {p^{k_1}-\frac 92 T_{k_1}}{p^{k_1/2}}>1.999998451823\quad\text{and}\quad \frac {p^{k_2}-\frac 92 T_{k_2}}{p^{k_2/2}}< -1.99999906282.$$
	\end{example}

Finally, we note that Corollary \ref{asinto3} improves  the constant $M_3=20$ in Theorem \ref{teoMasuda} to the value  2,   so, a natural question is how improve to $M_r$ for any $r\ge 4$.

\end{document}